\documentclass{amsart}
\title{Combinatorics of cycle lengths on Wehler K3 Surfaces over finite fields}
\author{Joao Alberto de Faria and Benjamin Hutz}
\date{}

\usepackage{color}
\usepackage{amssymb,amsmath,amsthm,fullpage}
\usepackage[all]{xy}  
\usepackage{url}
\usepackage{rotating} 
\usepackage{tikz}
\usepackage{textcomp,listings}
\usepackage{versions}

\definecolor{green}{rgb}{0,0.5,0}
\definecolor{dkgreen}{rgb}{0,0.6,0}
\definecolor{gray}{rgb}{0.5,0.5,0.5}
\definecolor{mauve}{rgb}{0.58,0,0.82}
\lstnewenvironment{python}[1][]{
\lstset{
  frame=single,                   
  language=python,                          
  basicstyle=\ttfamily\small, 
  numbers=left,                             
  numberstyle=\scriptsize\color{black},  
  stepnumber=1,                   
  numbersep=7pt,                  
  backgroundcolor=\color{white},      
  stringstyle=\color{red},
  showspaces=false,               
  showstringspaces=false,         
  showtabs=false,                 
  tabsize=2,                      
  captionpos=b,                   
  breaklines=true,                
  breakatwhitespace=false,        
  rulecolor=\color{black},        
  framexleftmargin=1mm, framextopmargin=1mm, frame=shadowbox, 
  alsoletter={1234567890},
  otherkeywords={\ , \}, \{},
  keywordstyle=\color{blue},       
  stringstyle=\color{mauve},         
  emph={access,and,break,class,continue,def,del,elif ,else,%
  except,exec,finally,for,from,global,if,import,in,i s,%
  lambda,not,or,pass,print,raise,return,try,while},
  emphstyle=\color{black}\bfseries,
  emph={[2]True, False, None, self},
  emphstyle=[2]\color{blue},
  emph={[3]from, import, as},
  emphstyle=[3]\color{blue},
  upquote=true,
  morecomment=[s]{"""}{"""},
  commentstyle=\color{dkgreen},       
  emph={[4]1, 2, 3, 4, 5, 6, 7, 8, 9, 0},
  emphstyle=[4]\color{blue},
  literate=*{:}{{\textcolor{blue}:}}{1}%
  {=}{{\textcolor{blue}=}}{1}%
  {-}{{\textcolor{blue}-}}{1}%
  {+}{{\textcolor{blue}+}}{1}%
  {*}{{\textcolor{blue}*}}{1}%
  {!}{{\textcolor{blue}!}}{1}%
  {(}{{\textcolor{blue}(}}{1}%
  {)}{{\textcolor{blue})}}{1}%
  {[}{{\textcolor{blue}[}}{1}%
  {]}{{\textcolor{blue}]}}{1}%
  {<}{{\textcolor{blue}<}}{1}%
  {>}{{\textcolor{blue}>}}{1},%
}}{}

\definecolor{orange}{rgb}{1,0.65,0.17}

\providecommand{\abs}[1]{\left\lvert#1\right\rvert}

\providecommand{\open}[1]{\textbf{Open:} \textcolor{red}{#1}}

\def\Q{\mathbb{Q}}
\def\P{\mathbb{P}}
\def\A{\mathbb{A}}

\def\F{\mathbb{F}}

\newcommand{\col}{\,{:}\,}
\newcommand{\tth}{^{\operatorname{th}}}

\DeclareMathOperator{\Fix}{Fix}

\theoremstyle{plain}
\newtheorem{thm}{Theorem}[section]
\newtheorem{lem}[thm]{Lemma}
\newtheorem{prop}[thm]{Proposition}
\newtheorem{cor}[thm]{Corollary}
\theoremstyle{definition}
\newtheorem*{defn}{Definition}

\newtheorem{exmp}{Example}[section]
\newtheorem*{exmp*}{Example}

\theoremstyle{remark}
\newtheorem*{rem}{Remark}

\excludeversion{code}

\begin{document}
\begin{abstract}
    We study the dynamics of maps arising from the composition of two non-commuting involution on a K3 surface. These maps are a particular example of reversible maps, i.e., maps with a time reversing symmetry. The combinatorics of the cycle distribution of two non-commuting involutions on a finite phase space was studied by Roberts and Vivaldi. We show that the dynamical systems of these K3 surfaces satisfy the hypotheses of their results, providing a description of the cycle distribution of the rational points over finite fields. Furthermore, we extend the involutions to include the case where there are degenerate fibers and prove a description of the cycle distribution in this more general situation.
\end{abstract}

\maketitle

\section{Introduction}
    This article examines the dynamics of a particular class of reversible maps that arise as automorphisms of a K3 surface. For a survey of time-reversing symmetries in dynamical systems, see \cite{LambRoberts}. We are interested mainly in the distribution of cycle lengths when the surface is defined over a finite field (i.e., when the map has a finite phase space). In the particular situation where the map is a composition of two involutions, as is the case for our K3 surfaces, Roberts-Vivaldi \cite{RobertsVivaldi} give a combinatorial description of the cycle distribution if the fixed points of the involutions satisfy certain properties. We demonstrate that their hypotheses hold for this class of K3 surfaces. Using an idea of Baragar \cite{Baragar}, we then extend the involutions on our K3 surfaces to the case when the surface has degenerate fibers (i.e. fibers of dimension 1).  We show that the fixed points of these extended involutions also satisfy the necessary properties, again yielding a combinatorial description of the cycle distribution from Roberts-Vivaldi \cite{RobertsVivaldi}.

\subsection{Reversible Dynamical Systems}
    We give a brief summary of definitions and results from the dynamics of reversible maps that are used in this article. Let $\phi: V \subseteq \mathbb{P}^N \to V$ be a morphism on a variety $V$.  We denote the $n\tth$ iterate of $\phi$ as
    \begin{equation*}
        \phi^n = \phi \circ \phi^{n-1}.
    \end{equation*}
    We say that $P \in V$ is a \emph{periodic point of period $n$ for $\phi$} if
    \begin{equation*}
        \phi^n(P) = P
    \end{equation*}
    and of \emph{minimal period $n$} if, in addition, for all $m < n$
    \begin{equation*}
        \phi^m(P) \neq P.
    \end{equation*}
    Such a map is called \emph{reversible} if there exists a map $R:V \to V$ such that
    \begin{equation*}
        R^{-1} \circ \phi \circ R = \phi^{-1}.
    \end{equation*}
    The map $R$ is called a \emph{reversor} for $\phi$.
    \begin{exmp}
        If $\phi$ is the composition of two involutions, $\phi = I_1 \circ I_2$, then $\phi$ is reversible since
        \begin{equation*}
            I_2^{-1} \circ \phi \circ I_2 = I_2 \circ I_1 = \phi^{-1}.
        \end{equation*}
    \end{exmp}
    \begin{defn}
        A cycle is \emph{symmetric} for $\phi = I_1 \circ I_2$ if it is invariant under $I_1$ (or $I_2$). Otherwise, the cycle is called \emph{asymmetric}.
    \end{defn}
    It is known that the number of symmetric cycles is determined by the number of fixed points of the involutions \cite{RobertsVivaldi2}. In particular, the number of symmetric cycles is given by
    \begin{equation*}
        \frac{\#\Fix(I_1) + \#\Fix(I_2)}{2}.
    \end{equation*}
    The following result on the distribution of cycle lengths was first conjectured by Roberts-Vivaldi \cite[Conjecture 1]{RobertsVivaldi} for polynomial automorphisms  of the plane. They were able to prove a more general statement several years later, which we recall below.

    Consider the following general combinatorial situation of the composition of two involutions on a set $S$ with $\phi=I_1 \circ I_2: S \to S$.
    \begin{defn}
        We define a distribution
        \begin{equation*}
            R_N(x) = \frac{1}{N}\{x \in S \col x \text{ has minimal period } \leq tz\},
        \end{equation*}
        where $\#S=N$ and
        \begin{equation*}
            z = \frac{2N}{\#\Fix(I_1) + \#\Fix(I_2)}
        \end{equation*}
        is a scaling parameter. Finally, define
        \begin{equation*}
            R(x) = 1-e^{-x}(1+x).
        \end{equation*}
    \end{defn}

    \begin{thm}[\textup{\cite[Theorem A]{RobertsVivaldi2}}] \label{thm_RV}
        Let $(I_1,I_2)$ be a pair of involutions on a set $S$ with $N$ points and let $i_1(N) = \#\Fix(I_1)$ and $i_2(N) = \#\Fix(I_2)$. If $i_1$ and $i_2$ satisfy
        \begin{equation*}
            \lim_{N \to \infty} i_1(N) + i_2(N) = \infty \qquad \lim_{N \to \infty} \frac{i_1(N) + i_2(N)}{N}=0,
        \end{equation*}
        then for all $x \geq 0$ we have
        \begin{equation*}
            \lim_{N \to \infty} R_N(x) = R(x).
        \end{equation*}
        Moreover, almost all points belong to symmetric cycles.
    \end{thm}

    \subsection{Wehler's K3 surfaces}
        We now define our particular dynamical system. A Wehler K3 surface $S \subset \mathbb{P}^{2}_x \times \mathbb{P}^{2}_y$ is
        a smooth surface given by the intersection of an effective divisor of degree (1,1) and an effective divisor of degree (2,2).  In other words,  let $([x_0,x_1,x_2],[y_0,y_1,y_2]) = (\textbf{x},\textbf{y})$ be the coordinates for $\mathbb{P}^2_x \times \mathbb{P}^2_y$; then
        $S$ is the locus described by $L=Q=0$ for
        \begin{equation*}
            L = \sum_{0\leq i,j \leq 2} a_{ij}x_iy_j \qquad
            Q = \sum_{0\leq i,j,k,l \leq 2} b_{ijkl}x_ix_jy_k y_l.
        \end{equation*}
        Wehler \cite{Wehler} first showed that these surfaces have an infinite automorphism group generated by the composition of two involutions.
        \begin{thm}[Wehler {\cite[Theorem 2.9]{Wehler}}] \label{thm1}
            A general K3 surface formed as the vanishing locus of a degree $(1,1)$ and a degree $(2,2)$ effective divisor has Picard number two and an infinite automorphism group.
        \end{thm}
        The involutions are defined as follows. The natural projections
        \begin{equation*}
            \rho_x: \mathbb{P}^{2}_x \times \mathbb{P}^2_y \to \mathbb{P}^2_x, \quad
            \rho_y: \mathbb{P}^{2}_x \times \mathbb{P}^2_y \to \mathbb{P}^2_y
        \end{equation*}
        induce two projection maps:
        \begin{equation*}
            p_x: S \to \mathbb{P}^2_x, \quad
            p_y: S \to \mathbb{P}^2_y.
        \end{equation*}
        The projections $p_x$ and $p_y$ are in general double covers, allowing us to define two involutions of $S$, say $\sigma_x$ and $\sigma_y$, respectively.  The maps $\sigma_x$ and $\sigma_y$ are in general just rational maps.  However, if $S=V(L,Q)$ is smooth, Call and Silverman \cite[Proposition 1.2]{CallSilverman} show that $\sigma_x$ and $\sigma_y$ are morphisms of $S$ if and only if $S$ has no degenerate fibers, fibers of positive dimension.  We call a surface with no degenerate fibers a \emph{non-degenerate} surface. Call and Silverman \cite[Appendix]{CallSilverman} give explicit formulas for computing $\sigma_x$ and $\sigma_y$ and, hence, any $\tau \in \mathcal{A}$.  We adopt their notation and define
        \begin{align*}
            L_j^{x} &= \text{ the coefficient of } y_j \text{ in } L(\textbf{x},\textbf{y}), \\
            L_j^{y} &= \text{ the coefficient of } x_j \text{ in } L(\textbf{x},\textbf{y}),\\
            Q_{kl}^{x} &= \text{ the coefficient of } y_ky_l \text{ in } Q(\textbf{x},\textbf{y}),\\
            Q_{ij}^{y} &= \text{ the coefficient of } x_ix_j \text{ in } Q(\textbf{x},\textbf{y}),\\
            G_k^{\ast} &= (L_{j}^{\ast})^2Q_{ii}^{\ast} - L_i^{\ast}L_j^{\ast}Q_{ij}^{\ast} + (L_i^{\ast})^2Q_{jj}^{\ast},\\
            H_{ij}^{\ast} &= 2L_i^{\ast}L_j^{\ast}Q_{kk}^{\ast} - L_i^{\ast}L_k^{\ast}Q_{jk}^{\ast} - L_j^{\ast}L_k^{\ast}Q_{ik}^{\ast} + (L_k^{\ast})^2Q_{ij}^{\ast}
        \end{align*}
        for $(i,j,k)$ some permutation of the indices $\{0,1,2\}$ and $\ast$ replaced by either $x$ or $y$.

        We take our dynamical system as $\phi = \sigma_y \circ \sigma_x$. Thus, for each smooth, non-degenerate Wehler K3 surface, we get a reversible dynamical system that is the composition of two involutions. The arithmetic and dynamical properties of these surfaces have received considerable attention in recent years, \cite{Baragar, Baragar2, CallSilverman, Hutz, Silverman}.

\section{Main Results}

\subsection{Distribution}
    Let $S$ be a Wehler K3 surface and $S(K)$ be the rational points on $S$ defined over the field $K$. We denote the finite field with $p$ elements as $\F_p$. The following definition sets up the distribution function of cycle lengths in a fashion similar to Roberts-Vivaldi \cite{RobertsVivaldi2}.
    \begin{defn}
        Let $S$ be a Wehler K3 surface. Define
        \begin{equation*}
            P_t = \frac{1}{\#S(\F_p)} \{P \in S(\F_p) \col P \text{ has minimal period } t\}.
        \end{equation*}
        For a given surface $S$, the sequence $P_t$ contains only finitely many non-zero terms. We consider the distribution function
        \begin{equation*}
            R_p(x) = \sum_{t =1}^{\lfloor xz \rfloor} \langle P_t \rangle,
        \end{equation*}
        where $\lfloor \cdot \rfloor$ represents the greatest integer part (the floor function), the average $\langle \cdot \rangle$ is computed with respect to uniform probability on the set of Wehler K3 surfaces $S$, and
        \begin{equation*}
            z = \frac{2 N}{\#\Fix(\sigma_x) + \#\Fix(\sigma_y)},
        \end{equation*}
        where $N$ is the average value of $\#S(\F_p)$.
    \end{defn}

    \begin{lem}\label{lem1}
        Let $S$ be a Wehler K3 surface. Then, for $p$ an integer prime,
        \begin{equation*}
            \#S(\F_p) \geq p^2-22p+1.
        \end{equation*}
    \end{lem}
    \begin{proof}
        Let $p \in \mathbb{Z}$ be a prime and let
        $N_m = \#S(\mathbb{F}_{p^m})$ be the number of $\mathbb{F}_{p^m}$-rational points on $S$. The \emph{Riemann zeta function} of $S$ is
        \begin{equation*}
            Z(S,T) = \exp\left(\sum_{m =1}^{\infty} N_m \frac{T^m}{m}\right),
        \end{equation*}
        where $\exp$ denotes exponentiation. The Riemann zeta function satisfies the Riemann Hypothesis, as shown by Deligne \cite{Deligne}.  We have that $\dim(S) =2$ and, since $S$ is a K3 surface it has Betti numbers $b_0=b_4=1$, $b_1=b_3=0$, and $b_2=22$. Thus, we have that
        \begin{align*}
            Z(S,T) &= \exp\left(\sum_{m =1}^{\infty} N_m \frac{T^m}{m}\right) =
            \frac{1}{P_0(T)P_2(T)P_{4}(T)}\\
            &=\frac{1}{(1-T)(1-p^2T)(\prod_{i=1}^{22}(1-\alpha_iT))}
        \end{align*}
        with $\abs{\alpha_i} =p$. We can take the natural logarithm of both sides, expand, and compare coefficients to get
        \begin{equation*}
            \#S(\mathbb{F}_p) = N_1 \geq 1 + p^2 - \sum_{i=1}^{22} \abs{\alpha_i} \geq 1 + p^2 - 22p.
        \end{equation*}
    \end{proof}

    \begin{thm}\label{thm2}
        Let $S$ be a non-degenerate Wehler K3 surface. We have
        \begin{equation*}
            \lim_{p \to \infty} R_p(x) = R(x) = 1 + e^{-x}(1+x).
        \end{equation*}
        Moreover, almost all cycles are symmetric.
    \end{thm}

    \begin{proof}
        Call and Silverman \cite[Proposition 2.1]{CallSilverman} describe the ramification curves $g^x$, $g^y$ as
        \begin{align*}
            g_{\ast} &= L_0^{\ast 2}Q_{12}^{\ast 2} + L_1^{\ast 2}Q_{02}^{\ast 2} + L_2^{\ast 2}Q_{01}^{\ast 2} -2L_0^{\ast}L_1^{\ast}Q_{02}^{\ast}Q_{12}^{\ast} - 2L_0^{\ast}L_2^{\ast}Q_{01}^{\ast}Q_{12}^{\ast} - 2L_1^{\ast}L_2^{\ast}Q_{01}^{\ast}Q_{02}^{\ast}\\
            &+ 4L_0^{\ast}L_1^{\ast}Q_{01}^{\ast}Q_{22}^{\ast} + 4L_0^{\ast}L_2^{\ast}Q_{02}^{\ast}Q_{11}^{\ast} + 4L_1^{\ast}L_2^{\ast}Q_{12}^{\ast}Q_{00}^{\ast} -4L_0^{\ast 2}Q_{11}^{\ast}Q_{22}^{\ast} - 4L_1^{\ast 2}Q_{00}^{\ast}Q_{22}^{\ast} - 4L_2^{\ast 2}Q_{11}^{\ast}Q_{00}^{\ast}.
        \end{align*}
        These are smooth degree 6 curves in $\P^2$ which describe the fixed points of the involutions $\sigma_x$ and $\sigma_y$. As such, we can apply the Hasse-Weil bounds for a genus $10$ curve to have that on $S(\F_p)$
        \begin{equation*}
            \abs{\#\Fix(\sigma_{\ast})-(p+1)} \leq 20\sqrt{p}.
        \end{equation*}
        In particular,
        \begin{equation} \label{eq_hasse1}
            (p+1)-20\sqrt{p} \leq \#\Fix(\sigma_{\ast}) \leq (p+1)+20\sqrt{p}.
        \end{equation}
        To apply Theorem \ref{thm_RV}, we need to show
        \begin{align}
            \lim_{p \to \infty} &\#\Fix(\sigma_x) + \#\Fix(\sigma_y) = \infty \label{eq2}\\
            \lim_{p \to \infty} &\frac{\#\Fix(\sigma_x) + \#\Fix(\sigma_y)}{\#S(\F_p)}=0. \label{eq3}
        \end{align}
        We consider the limit of the lower Hasse-Weil bound of (\ref{eq_hasse1}),
        \begin{equation*}
            \lim_{p \to \infty} (p+1)-20\sqrt{p} = \infty.
        \end{equation*}
        Therefore,
        \begin{equation*}
            \lim_{p \to \infty} \#\Fix(\sigma_x) + \#\Fix(\sigma_y) = \infty,
        \end{equation*}
        satisfying the first property (\ref{eq2}).

        To show that (\ref{eq3}) holds, we consider the upper Hasse-Weil bound of (\ref{eq_hasse1}) and a lower bound on $\#S(\F_p)$ from Lemma \ref{lem1},
        \begin{equation*}
            \lim_{p \to \infty} \frac{\#\Fix(\sigma_x) + \#\Fix(\sigma_y)}{\#S(\F_p)} < \lim_{p \to \infty} \frac{2((p+1)+20\sqrt{p})}{p^{2}-22p+1} = 0.
        \end{equation*}
        Since the fraction is always nonnegative, we have our result by applying Theorem \ref{thm_RV}.
    \end{proof}

    It is interesting to note that even for small primes, the actual distribution is extremely close to the limiting distribution. Figure \ref{figure1} shows the experimentally gathered cycle distributions (the dots) versus the limiting distribution $y=R(x) = 1+e^{-x}(1+x)$. Figure \ref{figure2} shows the error calculated from the difference in area under the curves. The $y$-axis is percent error $\abs{\frac{\text{actual value} - \text{experimental value}}{\text{actual value}}}$. The $x$-axis is $p$ for $\F_p$. The data is from 100 randomly generated Wehler K3 surfaces over $\F_p$ for
    \begin{equation*}
        p \in \{29, 37, 59, 61, 83, 113, 131, 149, 167, 181, 191, 223, 251, 269, 307, 353, 401, 457, 503\} \footnote{If the surface was degenerate for the listed prime, we used the next biggest good prime.}.
    \end{equation*}
    The computations were performed in Sage \cite{sage}.
    \begin{figure}
        \begin{center}
            \caption{Average Distribution}
            \label{figure1}
            \includegraphics[angle=0,scale=.45]{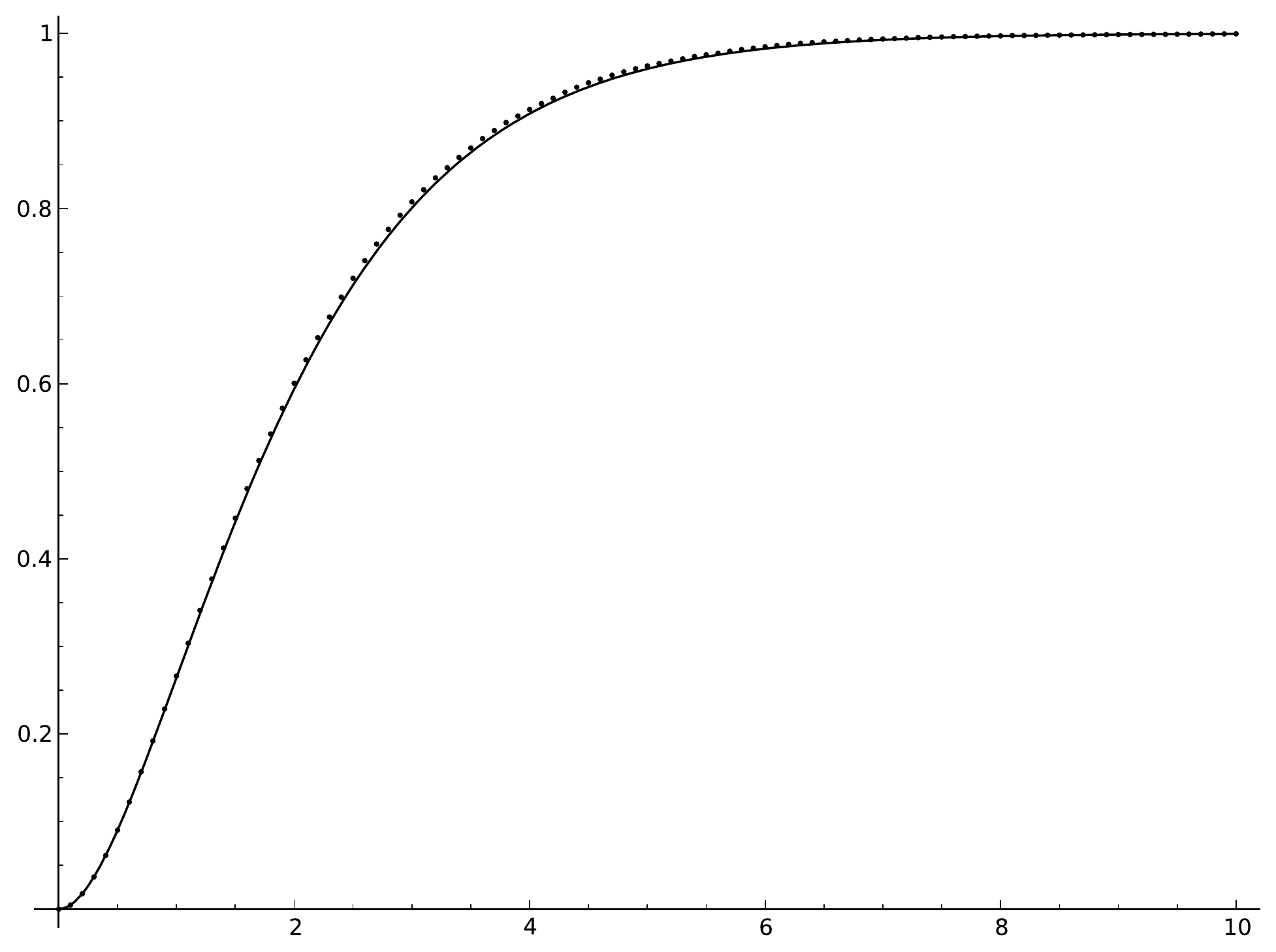}
        \end{center}
    \end{figure}

    \begin{figure}
        \begin{center}
            \caption{Error Data}
            \label{figure2}
            \includegraphics[angle=0,scale=.45]{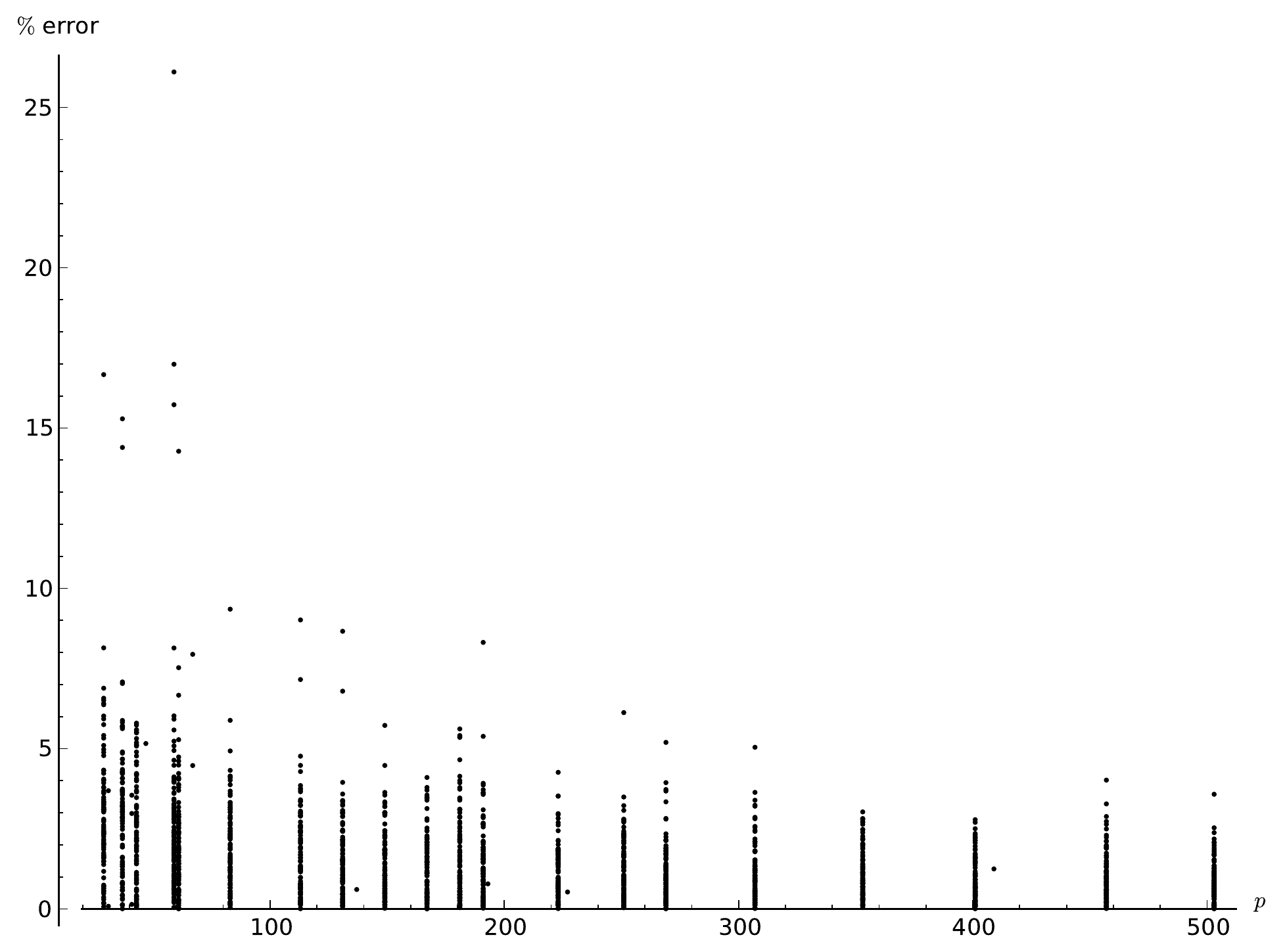}
        \end{center}
    \end{figure}

%
%
%
%

\subsection{Asymmetric Cycles}
    For the composition of two involutions, the reason the fixed points of the involutions play such a dominating role is that one of the points in the symmetric cycle must be a fixed point of each involution. For asymmetric cycles, this is not true and causes asymmetric cycles to always come in pairs. Figure \ref{figure5} gives a graphical representation as to why this is true.
    \begin{prop}
        Let $S$ be a Wehler K3 surface and $\phi=\sigma_{y}\circ\sigma_{x}$ the composition of the two involutions. All asymmetric cycles of $\phi$ of minimal period $n$ come in pairs.
    \end{prop}

    \begin{proof}
        Let $P$ be a point of minimal period $n$ for $\phi$.  Consider the point $Q = \sigma_x(P)$. By the assumption on $P$, we also have
        \begin{equation}\label{eq1}
            Q = \sigma_x(\phi^n(P)).
        \end{equation}
        Now recall that $\phi=\sigma_{y}\circ\sigma_{x}$ and expand $\phi^{n}$ in (\ref{eq1}) as
        \begin{equation*}
            Q=\sigma_{x}\circ(\sigma_{y}\circ\sigma_{x})\circ(\sigma_{y}\circ\sigma_{x})...(\sigma_{y}\circ\sigma_{x})(P).
        \end{equation*}
        Regrouping the compositions, we have
        \begin{equation*}
            Q=(\sigma_{x}\circ\sigma_{y})\circ(\sigma_{x}\circ\sigma_{y})...(\sigma_{x}\circ\sigma_{y})\circ(\sigma_{x} (P))=(\phi^{-1})^n(\sigma_x(P)) = (\phi^{-1})^n(Q).
        \end{equation*}
        Now we just need to check that $Q$ has minimal period $n$. Assume that $\phi^m(Q) =Q$ for some $m < n$. Then with the same argument in reverse, we would necessarily have $\phi^m(P) =P$, which is a contradiction.

        Therefore, given a periodic point $P$ with symmetric cycle of minimal period $n$, then $\sigma_x(P)$ is also periodic with a symmetric cycle of minimal period $n$.
    \end{proof}

   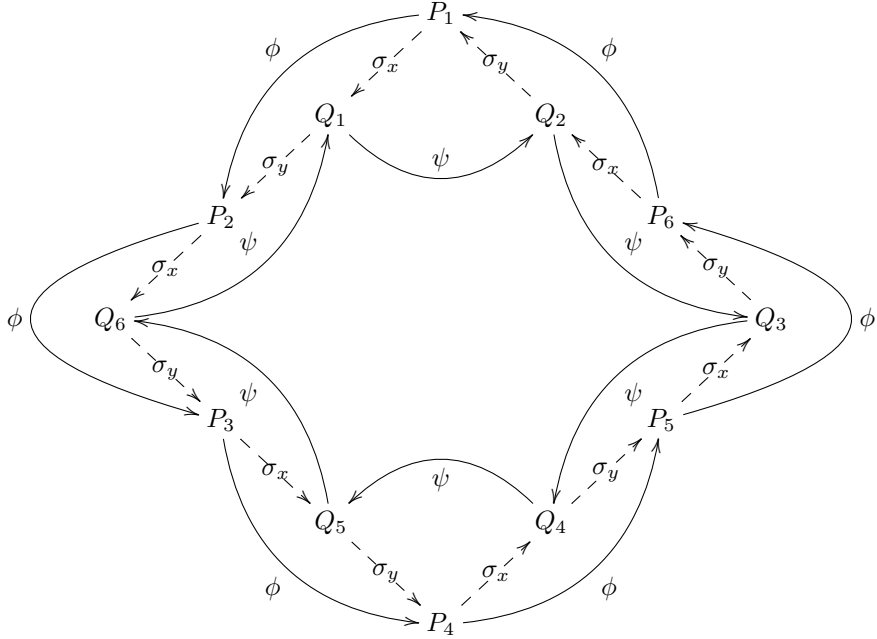
\begin{figure}
        \begin{center}
            \caption{Asymmetric 6-cycle}
            \label{figure5}
            \[
	 \renewcommand{\labelstyle}{\textstyle}
	 \xymatrix{
	 &&&P_1 \ar@{-->}[ld]|{\sigma_x} \ar@/_2pc/ [lldd] _\phi   \\
	&&Q_1 \ar@{-->}[ld]|{\sigma_y} \ar@/_2pc/ [rr] ^\psi  &&Q_2 \ar@{-->}[lu]|{\sigma_y} \ar@/_2pc/ [rrdd] ^\psi  \\
	&P_2 \ar@{-->}[ld]|{\sigma_x} \ar@/_6pc/ [dd] _\phi  &&&& P_6 \ar@{-->}[lu]|{\sigma_x} \ar@/_2pc/ [lluu] _\phi  \\
	Q_6 \ar@{-->}[rd]|{\sigma_y} \ar@/_2pc/ [rruu] ^\psi  &&&&&& Q_3 \ar@{-->}[lu]|{\sigma_y} \ar@/_2pc/ [lldd] ^\psi  \\
	&P_3 \ar@{-->}[rd]|{\sigma_x} \ar@/_2pc/ [rrdd] _\phi  &&&& P_5 \ar@{-->}[ru]|{\sigma_x} \ar@/_6pc/ [uu] _\phi  \\
	&&Q_5 \ar@{-->}[rd]|{\sigma_y} \ar@/_2pc/ [lluu] ^\psi  &&Q_4 \ar@{-->}[ru]|{\sigma_y} \ar@/_2pc/ [ll] ^\psi  \\
	&&&P_4 \ar@{-->}[ru]|{\sigma_x} \ar@/_2pc/ [rruu] _\phi  \\
	}
	\]
        \end{center}
    \end{figure}

\section{Degenerate Fibers}
    Following the idea sketched in Baragar \cite{Baragar}, we extend the morphisms $\sigma_x$ and $\sigma_y$ to degenerate fibers. The idea is to blow-up the surface at the degenerate points. This provides an isomorphism with the non-degenerate points and replaces the degenerate point with a family of lines. Each of those lines intersects the (blown-up) surface in two points, allowing for an extension of the involutions by again swapping points in the ``fibers.'' We now provide the necessary details.

    Let $P=(P_0,P_1,P_2)$ be a degenerate point on the second projection ($\P^2_y$), i.e., where the fiber $p_y^{-1}(P)$ is dimension 1.  After possibly a projective transformation, we may assume $P=(0,0,1)$. Dehomogenize at $y_2$ and consider the resulting locus $V(\tilde{L},\tilde{Q}) \subset \P^2 \times \A^2$. We label the coordinates of $\A^2$ as $(Y_0,Y_1)$. We blow-up the degenerate point by considering the lines through the origin (in $\A^2$).
    \begin{equation*}
        \tilde{X} = V(\tilde{L},\tilde{Q},s_1Y_0-s_0Y_1) \subset \P^2 \times (\A^2 \times \P^1),
    \end{equation*}
    where $s=(s_0,s_1)$ are the coordinates of $\P^1$.

    \begin{prop}\label{prop1}
        Each line through the degenerate point (in the blow-up) intersects the degenerate fiber in exactly two points.
    \end{prop}
    \begin{proof}
        Let $(P_0,P_1,P_2)$ be the degenerate point for the second projection, which after possibly a projective transformation we may assume is $(0,0,1)$. We dehomogenize to $(0,0)$ on $(Y_0,Y_1)$. For $(s_0,s_1) \in \P^1$, we take
        \begin{equation}\label{eq4}
            s_0Y_0 = s_1Y_1.
        \end{equation}
        We replace $(y_0,y_1,y_2)$ by $(Y_0,Y_1,1)$ and solve (\ref{eq4}) for $Y_0$ to get
        \begin{equation*}
            (y_0,y_1,y_2) = (s_1Y_1,s_0Y_1,s0).
        \end{equation*}
        We make this replacement to have
        \begin{equation}\label{eq5}
            G_i(Y_1,s)x_j^2 + H_{ij}(Y_1,s)x_ix_j + G_j(Y_1,s)x_i^2 = L_k^2Q + L(\text{some poly}).
        \end{equation}
        It is important to note that we now have the $G_i$ and $H_{ij}$ as functions of $s, Y_1$. At a point on the surface, the right-hand side is $0$ and at the degenerate point ($Y_1=0$), these coefficients $G_i,H_{ij}$ are identically $0$ \cite[Proposition 1.4]{CallSilverman}; so they are divisible by $Y_1$ (to some power). Dividing by the highest possible power of $Y_1$ we get a new version of (\ref{eq5}):
        \begin{equation}\label{eq6}
            G_i'(Y_1,s)x_j^2 + H_{ij}'(Y_1,s)x_ix_j + G_j'(Y_1,s)x_i^2 = 0.
        \end{equation}
        Each $s \in \P^1$ corresponds to a line through the origin, and for each $s$ there are $2$ sets of $x$ values solving (\ref{eq6}).
    \end{proof}
    Note that we may perform the same substitution, $(y_0,y_1,y_2) = (s_1Y_1,s_0Y_1,s_0)$, for $L$ and define $L'$ to be the result after dividing out by the highest possible power of $Y_1$. The following corollary generalizes \cite[Corollary 1.5]{CallSilverman} which is the key result for a practical algorithm to compute the involutions.
    \begin{cor}\label{cor1}
        Let $P=(a,b)$ be a point on $X$ and let $\sigma_y(P) = (a',b)$.

        \begin{enumerate}
            \item If $S_b$ is a non-degenerate fiber then $a,a'$ are the unique points on $L_b$ satisfying
                \begin{equation*}
                    G_k(b)x_l^2 + H_{kl}(b)x_kx_l + G_l(b)x_k^2 = 0 \qquad (k,l) \in \{(0,1),(0,2),(1,2)\}.
                \end{equation*}
                For each such pair $(k,l)$ the coordinates of $P, \sigma_y(P)$ satisfy
                \begin{equation} \label{eq11}
                    [a_ka_k',a_ka_l'+ a_la_k', a_la_l'] = [G_k(b), -H_{kl}(b), G_l(b)].
                \end{equation}
            \item \label{cor1_b} If $S_b$ is a degenerate fiber, then $a,a'$ are the unique points on $L_b'$, satisfying
                \begin{equation*}
                    G_k'(b,s)x_l^2 + H_{kl}'(b,s)x_kx_l + G_l'(b,s)x_k^2 = 0 \qquad (k,l) \in \{(0,1),(0,2),(1,2)\}.
                \end{equation*}
                For each such pair $(k,l)$ the coordinates of $P, \sigma_y(P)$ satisfy
                \begin{equation}\label{eq12}
                    [a_ka_k',a_ka_l'+ a_la_k', a_la_l'] = [G_k'(b,s), -H_{kl}'(b,s), G_l'(b,s)].
                \end{equation}
        \end{enumerate}
    \end{cor}
    \begin{proof}
        The first part is \cite[Corollary 1.5]{CallSilverman}.

        The second part is the same, but using equations (\ref{eq6}) from the proof of Proposition \ref{prop1}. We can examine the coefficients of these as polynomials in $x_i$ to show that the two roots $x_i,x_i'$ must satisfy
        \begin{equation*}
            \frac{x_i}{x_j}\frac{x_i'}{x_j'} = \frac{G_i'}{G_j'}
        \end{equation*}
        and
        \begin{equation*}
            \frac{x_i}{x_j}+\frac{x_i'}{x_j'} = \frac{x_ix_j' + x_i'x_j}{x_jx_j'} = \frac{-H_{ij}'}{G_j'}.
        \end{equation*}
    \end{proof}
    For non-degenerate fibers, equation (\ref{eq11}) of Corollary \ref{cor1} is used directly to compute the new point \cite{CallSilverman}, but we cannot do the same for degenerate fibers. The reason is that given a point on a degenerate fiber, we would need to know the unique $(s_0,s_1)$ associated to the point. What we can do is combine equation (\ref{eq12}) from Corollary \ref{cor1} with blown-up versions of $L$ and $Q$ to obtain a variety whose points are the two points in the ``fiber'' with the associated unique $(s_0,s_1)$ value. This procedure is detailed in the proceeding section.

    \subsection{Computing $\sigma_{\ast}$ on degenerate fibers}
        In practice, we do not need to move the degenerate point to $(0,0,1)$. Let $(y_0,y_1,y_2)=(P_0,P_1,P_2)$ be the $y$-coordinates of the degenerate point. Assuming that $P_2 \neq 0$, we dehomogenize to $(p_0,p_1)$ on $(Y_0,Y_1)$. For $(s_0,s_1) \in \P^1$, we take
        \begin{equation}\label{eq7}
            s_0(Y_0-p_0) = s_1(Y_1-p_1).
        \end{equation}
        We replace $(y_0,y_1,y_2)$ by $(Y_0,Y_1,1)$ and solve (\ref{eq7}) for $Y_0$ to get
        \begin{equation*}
            (y_0,y_1,y_2) = (s1(Y_1-p_1) + s_0p_0,s_0Y_1,s_0).
        \end{equation*}
        Dehomogenizing at a different coordinate gives a similar substitution.

        We make this replacement to have
        \begin{equation}\label{eq8}
            G_i(y_1,s)x_j^2 + H_{ij}(y_1,s)x_ix_j + G_j(y_1,s)x_i^2 = 0.
        \end{equation}
        At $Y_1=p_1$, these coefficients, $G_i,H_{ij}$, are identically $0$, so they are divisible by $(Y_1-p_1)$ (to some power). Dividing by the highest possible power of $(Y_1-p_1)$, we get a new version of (\ref{eq8}):
        \begin{equation*}
            G_i'(y_1,s)x_j^2 + H_{ij}'(y_1,s)x_ix_j + G_j'(y_1,s)x_i^2 = 0.
        \end{equation*}
        Again, we can solve for the two roots $x_i,x_i'$ in terms of $y_1,s$
        \begin{equation*}
            x_ix_i' = G_i' \qquad x_ix_j' + x_i'x_j = -H_{ij}'.
        \end{equation*}
        To compute $\sigma_{\ast}$ we know $(x_0,x_1,x_2)$ and $(y_1)$, so use the 6 equations
        \begin{align*}
            x_ix_i' &= G_i' \quad i \in \{0,1,2\}\\
            x_ix_j' + x_i'x_j &= -H_{ij}' \quad i,j \in \{0,1,2\}
        \end{align*}
        plus $L',Q'$ in the variables $(x_0',x_1',x_2')$ and $(s_0,s_1)$. Where We $L',Q'$ are obtained by performing the same substitution, $(y_0,y_1,y_2) = (s1(Y_1-p_1) + s_0p_0,s_0Y_1,s_0)$, on $L,Q$ and dividing by $(Y_1-p_1)$. These 8 equations results in 2 points, which are swapped by the involution. Note that if $(s_0,s_1) = (0,1)$, then we may again get identically $0$ equations for the points and must divide by appropriate powers of $s_0$ before solving for the two $x$ coordinate points.

        Away from the degenerate (blown-up) point, the blow-up map is an isomorphism, i.e., we have an isomorphism with $X = V(L,Q)$, so this is truly an \emph{extension} of the $\sigma_{\ast}$.

    \subsection{Ramification locus on degenerate fibers}
        We follow the construction of Call-Silverman \cite{CallSilverman}, but using the $G_k', H_{ij}'$ as defined in the previous section.
        Define
        \begin{equation*}
            g'_{\ast} = \frac{(H_{ij}')^2 - 4G_i'G_j'}{(L_k')^2}
        \end{equation*}
        which we can compute by the same substitution as above
        \begin{equation*}
            (y_0,y_1,y_2) = (s_1Y_1,s_0Y_1,s_0)
        \end{equation*}
        and cancelling out and powers of $Y_1$. This is the discriminant of the quadratic equations for the degenerate fiber from Corollary \ref{cor1}(\ref{cor1_b}) and from equation (\ref{eq5}) is independent of the choice of $(i,j,k)$. The result is a degree $6$ equation in $(s_0,s_1)$ in $\P^1$. So this is a hypersurface in $\P^1$ and has at most $6$ points (exactly $6$ when counted with multiplicity).
        \begin{prop}\label{prop2}
            Let $P=[a,b]$.
            \begin{enumerate}
                \item If $a$ is degenerate, then $g_x'(a,s) =0$ if and only if $\sigma_x(P) = P$
                \item If $b$ is degenerate, then $g_y'(b,s) =0$ if and only if $\sigma_y(P) = P$
            \end{enumerate}
        \end{prop}
        \begin{proof}
            This follows directly from Corollary \ref{cor1}.
        \end{proof}

\begin{code}
\open{Something is not working here when $s=0$. It is the ''extra'' factors of $s0$ that occur. Is this ''extra'' factor always $s0^4$?}
\begin{python}
PP.<x0,x1,x2,y0,y1,y2>=ProjectiveSpace_cartesian_product([2,2],QQ)
R=PP.coordinate_ring()
l=y0*x0 + y1*x1 + (y0 - y1)*x2
q=(y1*y0)*x0^2 + ((y0^2)*x1 + (y0^2 + (y1^2 - y2^2))*x2)*x0 + (y2*y0 + y1^2)*x1^2+ (y0^2 + (-y1^2 + y2^2))*x2*x1
X=WehlerK3Surface([l,q])
print X.is_smooth()
X.degenerate_fibers()
P=X([0,0,1,0,0,1])
self=X
BR=X.ambient_space().base_ring()
if P[1][0]!=0:
    S.<z0,z1,z2,s0,s1,w1>=PolynomialRing(BR,6)
    t=w1-BR(P[1][1]/P[1][0])
    t1=BR(P[1][1]/P[1][0])
    phi=R.hom([z0,z1,z2,s0,s0*w1,s1*t + s0*P[1][2]/P[1][0]],S)
    T=[phi(self.L),phi(self.Q),phi(self.Gpoly(0,0)),phi(self.Gpoly(0,1)),phi(self.Gpoly(0,2)), -phi(self.Hpoly(0,0,1)),-phi(self.Hpoly(0,0,2)),-phi(self.Hpoly(0,1,2))]
elif P[1][1]!=0:
    S.<z0,z1,z2,s0,s1,w1>=PolynomialRing(BR,6)
    t=w1-BR(P[1][0]/P[1][1])
    t1=BR(P[1][0]/P[1][1])
    phi=R.hom([z0,z1,z2,s0*w1,s0,s1*t + s0*P[1][2]/P[1][1]],S)
    T=[phi(self.L),phi(self.Q),phi(self.Gpoly(0,0)),phi(self.Gpoly(0,1)),phi(self.Gpoly(0,2)), -phi(self.Hpoly(0,0,1)),-phi(self.Hpoly(0,0,2)),-phi(self.Hpoly(0,1,2))]
elif P[1][2]!=0:
    S.<z0,z1,z2,s0,s1,w1>=PolynomialRing(BR,6)
    t=w1-BR(P[1][1]/P[1][2])
    t1=BR(P[1][1]/P[1][2])
    phi=R.hom([z0,z1,z2,s1*(t) + s0*P[1][0]/P[1][2],s0*w1,s0],S)
    T=[phi(self.L),phi(self.Q),phi(self.Gpoly(0,0)),phi(self.Gpoly(0,1)),phi(self.Gpoly(0,2)), -phi(self.Hpoly(0,0,1)),-phi(self.Hpoly(0,0,2)),-phi(self.Hpoly(0,1,2))]
maxexp=[]
for i in range(2,len(T)):
    e=0
    while (T[i]/t^e).subs({w1:t1})==0:
        e+=1
    maxexp.append(e)
e=min(maxexp)
print "E:",e
(phi(X.Ramification_gx())/w1^e).subs({w1:t1}).factor()
print ((phi(X.L)/w1).subs({w1:0})).subs({s0:0,s1:1}),((phi(X.Q)).subs({w1:0})/s0^2).subs({s0:0,s1:1})
print X.sigmaY(X([1,1,-1,0,0,1])) #s0=0 is not a fixed point!!!
\end{python}

Otherwise, seems correct
\begin{python}
PP.<x0,x1,x2,y0,y1,y2>=ProjectiveSpace_cartesian_product([2,2],QQ)
R=PP.coordinate_ring()
l=y0*x0 + y1*x1 + (y0 - y1)*x2
q=(y1*y0)*x0^2 + ((y0^2)*x1 + (y0^2 + (y1^2 - y2^2))*x2)*x0 + (y2*y0 + y1^2)*x1^2+ (y0^2 + (-y1^2 + y2^2))*x2*x1
l=y0*x0 + y1*x1 + (y0 - y1)*x2
q=(y1*y0 + y2^2)*x0^2 + ((y0^2 - y2*y1)*x1 + (y0^2 + (y1^2 - y2^2))*x2)*x0 + (y2*y0 + y1^2)*x1^2+ (y0^2 + (-y1^2 + y2^2))*x2*x1
X=WehlerK3Surface([l,q])
print X.is_smooth()
print X.degenerate_fibers()
P=X([1,-1,-2,-1,-1,1])
P=X([1,-1,2,-1,-1,1])
self=X
BR=X.ambient_space().base_ring()
if P[1][0]!=0:
    S.<z0,z1,z2,s0,s1,w1>=PolynomialRing(BR,6)
    t=w1-BR(P[1][1]/P[1][0])
    t1=BR(P[1][1]/P[1][0])
    phi=R.hom([z0,z1,z2,s0,s0*w1,s1*t + s0*P[1][2]/P[1][0]],S)
    T=[phi(self.L),phi(self.Q),phi(self.Gpoly(0,0)),phi(self.Gpoly(0,1)),phi(self.Gpoly(0,2)), -phi(self.Hpoly(0,0,1)),-phi(self.Hpoly(0,0,2)),-phi(self.Hpoly(0,1,2))]
elif P[1][1]!=0:
    S.<z0,z1,z2,s0,s1,w1>=PolynomialRing(BR,6)
    t=w1-BR(P[1][0]/P[1][1])
    t1=BR(P[1][0]/P[1][1])
    phi=R.hom([z0,z1,z2,s0*w1,s0,s1*t + s0*P[1][2]/P[1][1]],S)
    T=[phi(self.L),phi(self.Q),phi(self.Gpoly(0,0)),phi(self.Gpoly(0,1)),phi(self.Gpoly(0,2)), -phi(self.Hpoly(0,0,1)),-phi(self.Hpoly(0,0,2)),-phi(self.Hpoly(0,1,2))]
elif P[1][2]!=0:
    S.<z0,z1,z2,s0,s1,w1>=PolynomialRing(BR,6)
    t=w1-BR(P[1][1]/P[1][2])
    t1=BR(P[1][1]/P[1][2])
    phi=R.hom([z0,z1,z2,s1*(t) + s0*P[1][0]/P[1][2],s0*w1,s0],S)
    T=[phi(self.L),phi(self.Q),phi(self.Gpoly(0,0)),phi(self.Gpoly(0,1)),phi(self.Gpoly(0,2)), -phi(self.Hpoly(0,0,1)),-phi(self.Hpoly(0,0,2)),-phi(self.Hpoly(0,1,2))]
maxexp=[]
for i in range(2,len(T)):
    e=0
    while (T[i]/t^e).subs({w1:t1})==0:
        e+=1
    maxexp.append(e)
e=min(maxexp)
print "E:",e
print "ram:",(phi(X.Ramification_gx())/(w1-t1)^(2*e)).subs({w1:t1}).factor()
print "L,Q:",((phi(X.L)).subs({w1:t1})).subs({s0:-4,s1:3}),":",((phi(X.Q)).subs({w1:t1})).subs({s0:-4,s1:3})
print X.sigmaY(P)
Q=X.sigmaY(P)
X.sigmaY(Q)==Q, P==Q
\end{python}
\end{code}

    \begin{thm}
        Let $S$ be a (possibly degenerate) Wehler K3 surface. We have
        \begin{equation*}
            \lim_{p \to \infty} R_p(x) = R(x) = 1 + e^{-x}(1+x).
        \end{equation*}
        Moreover, almost all cycles are symmetric.
    \end{thm}
    \begin{proof}
        We apply the same proof as for Theorem \ref{thm2}, but have to take into account the fixed points of $\sigma_{\ast}$ that occur on degenerate fibers.

        On the non-degenerate fibers we apply the Hasse-Weil bounds for a genus $10$ curve and for each degenerate fiber there are at most $6$ fixed points
        \begin{equation*}
            \abs{\#\Fix(\sigma_{\ast})-(p+1)} \leq 20\sqrt{p} + 6w_p^{\ast},
        \end{equation*}
        where $w_p^{\ast}$ is the number of degenerate fibers in $S(\F_p)$. Let $w_0^{\ast}$ be the number of degenerate fibers in $S(\Q)$.
        In particular,
        \begin{equation} \label{eq_hasse2}
            (p+1)-20\sqrt{p} \leq \#\Fix(\sigma_{\ast}) \leq (p+1)+20\sqrt{p} + 6w_p.
        \end{equation}
        To apply Theorem \ref{thm_RV}, we need to show
        \begin{align}
            \lim_{p \to \infty} &\#\Fix(\sigma_x) + \#\Fix(\sigma_y) = \infty \label{eq9}\\
            \lim_{p \to \infty} &\frac{\#\Fix(\sigma_x) + \#\Fix(\sigma_y)}{\#S(\F_p)}=0. \label{eq10}
        \end{align}
        We consider the limit of the lower Hasse-Weil bound of (\ref{eq_hasse2}),
        \begin{equation*}
            \lim_{p \to \infty} (p+1)-20\sqrt{p} = \infty.
        \end{equation*}
        Therefore
        \begin{equation*}
            \lim_{p \to \infty} \#\Fix(\sigma_x) + \#\Fix(\sigma_y) = \infty
        \end{equation*}
        satisfying the first property (\ref{eq9}).

        To show that (\ref{eq10}) holds, we consider the upper Hasse-Weil bound of (\ref{eq_hasse2}) and a lower bound on $\#S(\F_p)$ from Lemma \ref{lem1},
        \begin{equation*}
            \lim_{p \to \infty} \frac{\#\Fix(\sigma_x) + \#\Fix(\sigma_y)}{\#S(\F_p)} < \lim_{p \to \infty} \frac{2((p+1)+20\sqrt{p}) + 6(w_p^x + w_p^y)}{p^{2}-22p+1} = 0
        \end{equation*}
        since for almost all primes $w_p^{\ast} = w_0^{\ast}$. Since the fraction is always nonnegative, we have our result by applying Theorem \ref{thm_RV}.
    \end{proof}

\begin{exmp}
    Consider the Wehler K3 surface defined by
    \begin{align*}
        L&: x_0y_0 + x_1y_1 + x_2y_2\\
        Q&: x_1^2y_0^2 + 2x_2^2y_0y_1 + x_0^2y_1^2 - x_0x_1y_2^2.
    \end{align*}
    This has two degenerate fibers:
    \begin{equation*}
        p^{-1}_x (-1,-1,1) \qquad \text{and} \qquad p^{-1}_x(1,1,1).
    \end{equation*}
    The surface has an asymmetric $8$-cycle which includes a points in a degenerate fiber starting at the point
    \begin{equation*}
        [(-1,-1,1),(1,0,1)].
    \end{equation*}
\begin{code}
\begin{python}
#asymmetric 8-cycle with degenerate fibers
PP.<x0,x1,x2,y0,y1,y2>=ProjectiveSpace_cartesian_product([2,2],QQ)
l=y0*x0 + y1*x1 + y2*x2
q=x1^2*y0^2+2*x2^2*y0*y1+x0^2*y1^2-x0*x1*y2^2
X=WehlerK3Surface([l,q])
P=X([1,1,-1,1,0,1])
T=X.orbit_phi(P,8)
print T
X.is_symmetric_orbit(T)
[(-1 : -1 : 1 : 1 : 0 : 1), (1 : 0 : 1 : -1 : 2 : 1), (1 : 1 : 1 : -1 :
0 : 1), (-1 : 0 : 1 : 1 : -2 : 1), (-1 : -1 : 1 : 1 : 0 : 1), (1 : 0 : 1
: -1 : 2 : 1), (1 : 1 : 1 : -1 : 0 : 1), (-1 : 0 : 1 : 1 : -2 : 1), (-1
: -1 : 1 : 1 : 0 : 1)]
\end{python}
\end{code}
\end{exmp}

    We generate similar experimental data as for Figure \ref{figure1} and Figure \ref{figure2} for degenerate surfaces. Figure \ref{figure3} shows the experimentally gathered cycle distributions (the dots) versus the limiting distribution $y=R(x) = 1+e^{-x}(1+x)$. Figure \ref{figure4} shows the error calculated from the difference in area under the curves. The $y$-axis is percent error $\abs{\frac{\text{actual value} - \text{experimental value}}{\text{actual value}}}$. The $x$-axis is $p$ for $\F_p$. The data is from 100 randomly generated degenerate Wehler K3 surfaces over $\F_p$ for
    \begin{equation*}
        p \in \{29, 37, 59, 61, 83, 113, 131, 149, 167, 181, 191, 223, 251, 269, 307, 353, 401, 457, 503\} \footnote{If the surface was degenerate for the listed prime, we used the next biggest good prime.}.
    \end{equation*}
    The computations were performed in Sage \cite{sage}.
    \begin{figure}
        \begin{center}
            \caption{Average Distribution}
            \label{figure3}
            \includegraphics[angle=0,scale=.45]{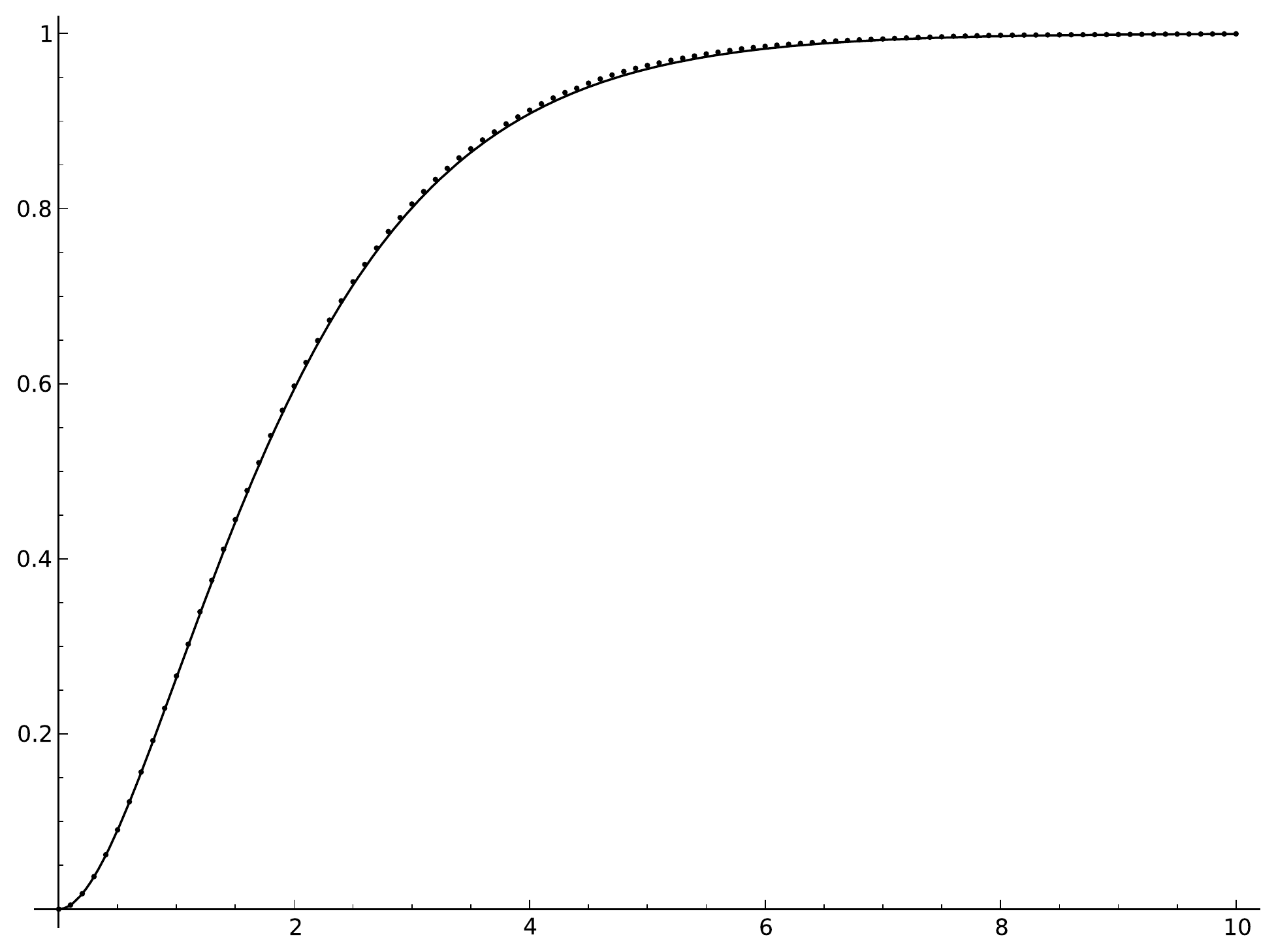}
        \end{center}
    \end{figure}

    \begin{figure}
        \begin{center}
            \caption{Error Data}
            \label{figure4}
            \includegraphics[angle=0,scale=.45]{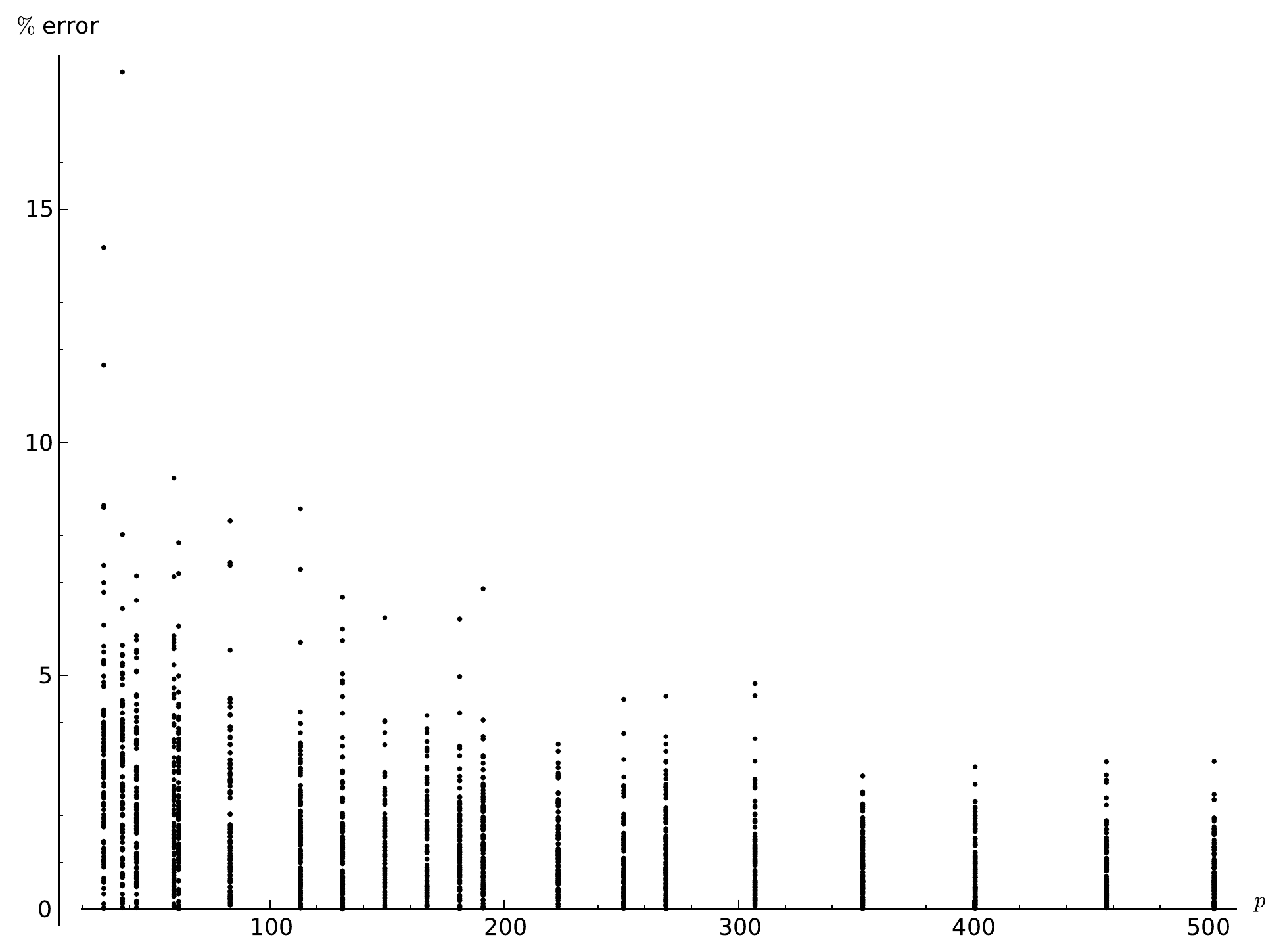}
        \end{center}
    \end{figure}

    \begin{rem}
        All algorithms used in this article are being written for inclusion in Sage \cite{sage} along with the algorithms in \cite{CallSilverman}.
    \end{rem}

\begin{code}
Code for average graphs
\begin{python}
load "/media/sf_WehlerK3/projective_space_cartesian_product.sage"

#average over all data
R.<x0,x1,x2,y0,y1,y2>=PolynomialRing(QQ,6)
M=10 #x from 0 to 10
N=100 #number of points

Surfaces=[]
count=0
for j in range(0,100):
    cycledata=[]
    #f = file('/media/sf_WehlerK3/Cycle_Lengths/Surface' + str(j) + '.txt','r')
    f = file('/media/sf_WehlerK3/Cycle_Lengths_degenerate/Surface' + str(j) + '.txt','r')
    for line in f:
        cycledata.append(sage_eval(line,locals=R.gens_dict()))
    f.close()


    [L,Q]=cycledata[0]
    perioddata=cycledata[1::]
    Points=[]
    for p,sym,asym in perioddata:
        T=[]
        S=sum(sym) #symmetric cycles
        x=0
        #z=p + log(p).n() #average value from my thesis data
        #exact value
        z=2*(sum(sym)+sum(asym))/(2*len(sym))
        z= S/len(sym) #only symmetric average length
        while x <= M:
            s=0
            for g in sym: #symmetric cycles
                if g <= z*x:
                    s+=g
            T.append([x,s/S.n()])
            x+=RR(ZZ(M)/ZZ(N))
        Points.append([p,T])
    Surfaces.append(Points)

#average
E=[[0,0] for i in range(N+1)]
for P in Surfaces:
    count+=len(P)
    for i in range(N+1):
        E[i][0]=P[0][1][i][0]
        for n in range(len(P)):
            E[i][1]+=P[n][1][i][1]
E=[[E[i][0],(E[i][1]/count).n()] for i in range(N+1)]
#Trapezoidal Rule
f(x) = 1-(1+x)/exp(x)
S=0
for i in range(1,len(E)):
    S+=1/2*(E[i][0]-E[i-1][0])*(E[i][1]+E[i-1][1])
S2=integrate(f,(x,0,M)).n()
S,S2, (S-S2).abs()

f(x) = 1-(1+x)/exp(x)
from sage.plot.point import Point
g1=plot(f(x),(x,0,10),color='black')
g= point(E,color='black',pointsize=6)
(g+g1).save(filename='average_data_degenerate.pdf')
(g+g1).show()
\end{python}

Code to generate error graphs
\begin{python}
load "/media/sf_WehlerK3/projective_space_cartesian_product.sage"

#Compute Error for each surface
R.<x0,x1,x2,y0,y1,y2>=PolynomialRing(QQ,6)
M=10 #x from 0 to 10
N=100 #number of points

ps=[29,31,37, 41, 43,47,59, 61, 67,83, 113, 131, 137, 149, 167, 181, 191, 193,223, 227,251, 269, 307, 353, 401, 409,457, 503]
errors=[[] for i in range(len(ps))]
for j in range(0,100):
    print"Surface ", j
    cycledata=[]
    #f = file('/media/sf_WehlerK3/Cycle_Lengths/Surface' + str(j) + '.txt','r')
    f = file('/media/sf_WehlerK3/Cycle_Lengths_degenerate/Surface' + str(j) + '.txt','r')
    for line in f:
        cycledata.append(sage_eval(line,locals=R.gens_dict()))
    f.close()


    [L,Q]=cycledata[0]
    perioddata=cycledata[1::]
    for p,sym,asym in perioddata:
        T=[]
        S=sum(sym) #symmetric cycles
        x=0
        #z=p + log(p).n() #average value from my thesis data
        #exact value
        z=2*(sum(sym)+sum(asym))/(2*len(sym))
        z=S/len(sym)  #only symmetric average length
        while x <= M:
            s=0
            for g in sym: #symmetric cycles
                if g <= z*x:
                    s+=g
            T.append([x,s/S.n()])
            x+=RR(ZZ(M)/ZZ(N))
        E=T
        #Trapezoidal Rule
        f(x) = 1-(1+x)/exp(x)
        S=0
        for i in range(1,len(E)):
            S+=1/2*(E[i][0]-E[i-1][0])*(E[i][1]+E[i-1][1])
        S2=integrate(f,(x,0,M)).n()
        #print p,(100*(S-S2)/S).abs()
        errors[ps.index(p)].append((100*(S-S2)/S).abs())
    print "-------------------------"

#arrange for graph
points=[]
for i in range(len(ps)):
    for t in errors[i]:
        points.append([ps[i],t])

#plot errors
g= point(points,color='black',pointsize=5,axes_labels=['$p$','$\%$ error'])
g.save(filename='error_data_degenerate.pdf')
g.show()
\end{python}

Code to generate cycle data
\begin{python}
load("~/Dropbox/SharedFolders/WehlerK3/projective_space_cartesian_product.sage")

@fork
def doit(p,Xp):
    g=Xp.Surgraph().all_simple_cycles()
    sym=[]
    asym=[]
    for c in g:
        if Xp.is_symmetric_orbit(c):
            sym.append(len(c)-1)
        else:
            asym.append(len(c)-1)
    return([p,sym,asym])

%time
filename1="/Users/bhutz/Dropbox/SharedFolders/WehlerK3/Surfaces.txt"
PP.<x0,x1,x2,y0,y1,y2>=ProjectiveSpace_cartesian_product([2,2],QQ)
with open(filename1) as f:
    Surfaces=f.readlines()
for i in range(len(Surfaces)):
    Surfaces[i]=Surfaces[i].replace("\n","")
    Surfaces[i]=eval(Surfaces[i].replace("^","**"))

for i in range(82,91):
    print i
    filename="/Users/bhutz/Dropbox/SharedFolders/WehlerK3/Cycle_Lengths/Surface" + str(i) + ".txt"

    plist=[29,37,43,59,61,83,113,131,149,167,181,191,223,251,269,307,353,401,457,503]
    #plist=[29,37,43,59,61,83,113,131,149,167,181,191,223,251,269]
    #plist=[307,353,401,457,503]
    pdone=[]
    L=[]

    #with open(filename) as g:
    #    oldvalues=g.readlines()
    #for j in range(len(oldvalues)):
    #    oldvalues[j]=oldvalues[j].replace("\n","")
    #g.close()

    X=WehlerK3Surface(Surfaces[i])
    #badprimes=X.degenerate_primes()
    #print badprimes
    for n in plist:
        p=n
     #   while p in badprimes:
      #      p=next_prime(p)
        Xp=X.change_ring(GF(p))
        if p not in pdone and Xp.is_smooth():
            pdone.append(p)
            L.append(doit(p,Xp))

    f = file(filename,'w')
    f.write(str([X.L,X.Q]))
    f.write('\n')
    for l in L:
        f.write(str(l))
        f.write('\n')
    #for l in oldvalues:
    #    f.write(l)
    #    f.write('\n')
    f.close()
\end{python}
\end{code}

\bibliography{biblio}
\bibliographystyle{plain}

\end{document}